\documentclass[a4paper, 11pt]{article}

\usepackage{amsmath,amssymb,amsthm}
\usepackage[latin1]{inputenc}
\usepackage{version,tabularx,multicol}
\usepackage{graphicx,float,psfrag}
\usepackage{stmaryrd}
 \usepackage{enumerate}
 \usepackage{color}
\usepackage[pdftex]{hyperref}
\usepackage{authblk}

\usepackage[numbers,sort&compress]{natbib}

\theoremstyle{plain}
\newtheorem{theorem}{Theorem}[section]

 \newtheorem{corollary}[theorem]{Corollary}

 \newtheorem{proposition}[theorem]{Proposition}

 \newtheorem{lemma}[theorem]{Lemma}

\newtheorem{remark}{Remark}[section]

  \makeatletter
  \@addtoreset{equation}{section}
  \makeatother

 \def\beqlb{\begin{eqnarray}}\def\eeqlb{\end{eqnarray}}
 \def\beqnn{\begin{eqnarray*}}\def\eeqnn{\end{eqnarray*}}

 \def\qed{\hfill$\Box$\medskip}

\newcommand{\bcen}{\begin{center}}
\newcommand{\ecen}{\end{center}}
\newcommand{\bgeqn}{\begin{equation}}
\newcommand{\edeqn}{\end{equation}}

\begin{document}

\title{Limit theorems for the critical Galton-Watson processes with  immigration stopped at zero}

\author[a]{Doudou Li \thanks{lidd@bjut.edu.cn}}
\author[b]{Mei Zhang \thanks{Corresponding author: meizhang@bnu.edu.cn.}}
\author[c]{Xianyu Zhang\thanks{xyzhang1995@126.com}}
\affil[a]{College of Statistics and Data Science, Faculty of Science, Beijing University of Technology,
 Beijing 100124, China}
\affil[b,c]{Laboratory of Mathematics and Complex Systems (Ministry of Education), School of Mathematical Sciences, Beijing Normal University, Beijing 100875, China}

\maketitle

\noindent\textit{Abstract:}
In this paper, we consider a critical Galton-Watson branching process with immigration stopped at zero $\mathbf{W}$. Some precise estimation on the generation function of the $n$-th population are obtained, and the  tail probability of the   life period of $\mathbf{W}$ is studied. Based on above results, two conditional limit theorems for $\mathbf{W}$ are established.
\bigskip


\bigskip


\section{Introduction}
 Let $\mathbf{Z}=\left\{
Z_{n},\ n\geq 0\right\} $ be a Galton-Watson branching process with immigration defined by
\begin{align}\label{Def Zn}
Z_{n}=\sum_{i=1}^{Z_{n-1}}\xi_{ni}+Y_n,\quad n\geq 1,\quad Z_0=Y_0,
\end{align}
where $Y_{0}$ is  a non-negative integer random variable with generating function  $G_{0}(s)$, $\{\xi_{ni},n,i\geq 1\}$ is a sequence of independent and identically distributed (i.i.d.) random variables with generating function $F(s)=\sum_{i=0}^\infty p_is^i$; \{$Y_n,n\geq 1$\} is another sequence of i.i.d. random variables with  generating function $B(s)=\sum_{i=0}^\infty b_is^i$, independent of $\{\xi_{ni},~n,i\geq1\}$. Denote $m:=\mathbb{E}\xi_{11}$. $\mathbf{Z}=\{Z_n,n\geq 0$\} is called supercritical, critical or subcritical, if $m>1$, $m=1$ or $m<1$. In this paper, we focus on the critical case.

From \eqref{Def Zn}, the generating function of $Z_n$ is of the form
\begin{align}\label{DefHn}
H_{n}(x)=G_{0}(F_{n}(x))\prod_{j=0}^{n-1}B[F_{j}(x)],\quad n\ge 1,
\end{align}
where $F_{j}(x)$ denotes the $j$-th iteration of the function $F(x)$ and $F_{0}(x)=x$.

Define
\begin{align}\label{Dlife}
\zeta:=\min\{n\geq 0:Z_{n}=0\},
\end{align}
then $\zeta$ is called the life period of the process $\mathbf{Z}$.

If we assume that $Y_{n}$ immigrants join the population if and only if  the individuals at generation $n-1$ have at least one offspring, then we can consider a modified version $\mathbf{W}=\left\{
W_{n},\ n\geq 0\right\} $ of the process $\mathbf{Z}$ specified
as follows. Without loss of generality, assume that $Z_{0}>0.$ Let
$W_{0}=Z_{0}$ and for $n\geq 1$,
\begin{align}\label{DefWn}
W_{n}:=
\begin{cases}
0, & \mbox{if}~\Lambda_{n}:=\xi
_{n1}+\ldots +\xi _{nZ_{n-1}}=0, \\
\Lambda_{n}+Y_{n}, & \mbox{if}~\Lambda_{n}>0.
\end{cases}
\end{align}
We call $\mathbf{W}$ as a branching process with immigration stopped at zero (BPIZ).  Recalling \eqref{Dlife}, we can see that
\begin{align*}
\zeta=\min\{n\geq 0:W_{n}=0\}.
\end{align*}

The distribution of the so-called life periods is important and interesting from not only theoretical but also practical points of view, see \cite{BM83}, \cite{DL20}, \cite{LVZ20}, \cite{Z72}. Information on the length of such periods may be used in epidemiology, ecology and seismology, see \cite{HJV}, \cite{Kag2010}, \cite{KA2002}. Meanwhile,
limit theorems for the critical branching processes with immigration are extensively studied, see for example, \cite{I85}, \cite{P72,P75a,P75b,P75}, \cite{V77}. By supposing $\mathbb{E}\xi_{11}^2<\infty$, $\mathbb{E}Y_1<\infty$ and $G_0(s) \equiv B(s)$, Zubkov~\cite{Z72} studied the process $\mathbf{W}$ defined by \eqref{DefWn}, and   studied the tail probability of $\zeta$ with the help of Tauberian theorem; Under the same assumption, Vatutin~\cite{V77} established  some scaling limit theorem for $\mathbf{W}$ conditioning on $W_n>0$. By assuming
\begin{align}\label{slack}
F(s)=s+(1-s)^{1+\nu}\mathcal{L}\left(\frac{1}{1-s}\right),\quad \nu\in(0,1],\end{align}
with some slowly varying function $\mathcal{L}$  at $\infty$, Pakes~\cite{P75b} proved some limit theorems for the population size of $\mathbf{Z}$. By assuming
$\mathbb{E}\xi_{11}^2<\infty$, $\mathbb{E}Y_1<\infty$ and $$
G_0(s)=1-(1-s)^\delta \mathcal{M}\left(\frac{1}{1-s}\right),\quad \delta\in(0,1], $$
with some slowly varying function $\mathcal{M}$ at $\infty$, Ivanoff and Seneta~\cite{I85} obtained  a conditional limit theorem for $\mathbf{W}$, improving the result in \cite{V77}.

In the present  paper, we are interested in the scaling limits of $\mathbf{W}$ conditioned on non-extinction, with the following assumptions on the generating functions of offspring and immigration:\\

\noindent\textbf{Basic Assumptions}
\begin{enumerate}
\item[(A1)] \quad $F(s)=s+\kappa_{1}(1-s)^{1+\nu},\quad \nu\in(0,1], $
\item[(A2)] \quad $B(s)=\mathrm{e}^{-\kappa_{2}(1-s)^{\theta}},\qquad\quad\quad \theta\in(0,1],$
\item[(A3)]\quad $G_{0}(s)=1-\kappa_{0}(1-s)^{\delta},\quad\; \delta\in(0,1], $
\end{enumerate}
where $\kappa_{i}(i=0,1,2)$ are positive constants.

First, we shall concentrate on $F_n(t)$,  the generating function of $Z_n$ without immigration, i.e. $Y_{n}\equiv0$. As we know, under the condition \eqref{slack} with $\nu=1$ (which means $\mathbb{E}\xi_{11}^2<\infty$), Kesten, Ney and Spitzer~\cite[Theorem 1]{KNS66} stated that the converging speed of $1/(1-F_n(t))$ to $1/(1-t)$ is $F''(1)n/2$. For $\nu\in (0,1)$, it is known from \cite{S68} that under \eqref{slack}, $(1-F_{n}(0))^{\nu}$ is of order $1/n$.
Later, Pakes~\cite[Lemma 1]{P75b} discussed $1-F_n(t)$, but did not give explicit estimation for $1-F_n(t)$  as $n\to \infty$.

In the present paper  (see Theorem~\ref{LFuniform}), we  shall prove that under Assumption (A1), $1/(1-F_n(t))^\nu$  converges to $1/(1-t)^\nu$ with speed $\kappa_1\nu n$,  generalizing \cite[Theorem 1]{KNS66}. Basing on it,  we then  study the asymptotic behavior of $\mathbb{P}(\zeta>n)$.  After that, some scaling limit theorems of $\mathbf{Z}$ are obtained,  generalizing  \cite[Lemmas 1-3]{V77} and  can be compared to   \cite[Theorems 3 and 5]{P75b}. With the help of above results, we finally get two conditional limit theorems of $\mathbf{W}$, which can be seen as partially improvement of Ivanoff and Seneta~\cite{I85} by  extending the assumptions $\mathbb{E}\xi_{11}^2<\infty$ and $\mathbb{E}Y_1<\infty$ to a more general case (see Assumptions (A1)--(A2)).

  We should mention that our results may hold under more general conditions \eqref{slack} and $B(s)= 1-\kappa_{2}(1-s)^{\theta}$, $\theta\in(0,1]$.  However, due to technical restriction,  here we only obtain  proofs of the main results under Assumptions (A1) and (A2).


In the following context, $C$ denotes a positive constant whose value may change from place to place.
With $f(n)=o(g(n))$ as $n\rightarrow \infty$, we refer that $\lim_{n\to \infty} f(n)/g(n)=0$; With $f(n)\asymp g(n)$ as $n\rightarrow \infty$, we refer that there exist  $M_1,M_2>0$ such that $M_1\le\liminf_{n\rightarrow \infty}f(n)/g(n)\le \limsup_{n\rightarrow \infty}f(n)/g(n)\leq M_2$.

The rest of this paper is organized as follows. In Section 2  we give a  precise estimation of $F_{n}(t)$. Life period of $\mathbf{W}$ is studied in Section 3.  In Section 4 the scaling limit theorems of $\mathbf{Z}$ are proved. Section 5 is devoted to the conditional limit theorems of $\mathbf{W}$.


\section{A precise estimation of $F_{n}(t)$}

According to \cite[Lemma 2]{S68} and Assumption (A1), we have
\begin{align*}
(1-F_{n}(0))^\nu \sim  \frac{1}{\kappa_{1}\nu n},\quad n\rightarrow\infty.
\end{align*}
In this section, we shall give some estimations on $1-F_{n}(t)$, see Theorem~\ref{LFuniform} and Proposition~\ref{l2.2}. Theorem~\ref{LFuniform} can be seen as a generalization of \cite[Theorem 1]{KNS66}. Both are important tools in the proofs of the scaling limit of $\mathbf{Z}$ in Section 4, and conditional limit theorems of $\mathbf{W}$ in Section 5.
\begin{theorem}\label{LFuniform}
Suppose  (A1) holds. Then
\begin{align*}
\lim\limits_{n\rightarrow\infty}\frac{1}{n}\left[\frac{1}{(1-F_{n}(t))^{\nu}}-\frac{1}{(1-t)^{\nu}}\right]
=\kappa_{1}\nu
\end{align*}
holds uniformly for $t\in[0,1)$.
\end{theorem}

\begin{proof}
For $t\in[0,1)$, define
\begin{align}\label{Theta}
\Theta(t):=\kappa_{1}\nu-\frac{(1-t)^{\nu}-(1-F(t))^{\nu}}{(1-F(t))^{2\nu}},
\end{align}
and
\begin{align*}
\Xi(t):=\kappa_{1}\nu-\left[\frac{1}{(1-F(t))^{\nu}}-\frac{1}{(1-t)^{\nu}}\right].
\end{align*}
Then
\begin{align}\label{Defups}
\Upsilon_{n}&:=\kappa_{1}\nu n-\left[\frac{1}{(1-F_{n}(t))^{\nu}}-\frac{1}{(1-t)^{\nu}}\right]\nonumber\\
&=\kappa_{1}\nu n-\sum\limits_{i=0}^{n-1}\left[\frac{1}{(1-F_{i+1}(t))^{\nu}}-\frac{1}{(1-F_{i}(t))^{\nu}}\right]\nonumber\\
&=\sum\limits_{i=0}^{n-1}\Xi(F_{i}(t)).
\end{align}

On the one hand, one can verify that $\Theta(t)$ is increasing, and $\Theta(t)\uparrow 0$ as $t\uparrow1$. Observing that  $\Theta(t)\leq \Xi(t)$ for $t\in[0,1)$, then
\begin{align}\label{lo}
\Upsilon_{n}\geq \sum\limits_{i=0}^{n-1}\Theta(F_{i}(t))
\geq \sum\limits_{i=0}^{n-1}\Theta(F_{i}(0))=o(n),
\end{align}
the last equality holds by the fact that $F_{n}(0)\rightarrow1$  as $n\rightarrow\infty$.

On the other hand, by the definition of $\Xi(t)$ and the  monotonicity of $\Theta(t)$, we have
\begin{align*}
\Xi(t)&=\left[\kappa_{1}\nu\frac{(1-t)^{\nu}}{(1-F(t))^{\nu}}
-\frac{(1-t)^{\nu}-(1-F(t))^{\nu}}{(1-F(t))^{2\nu}}\right]\frac{(1-F(t))^{\nu}}{(1-t)^{\nu}}\\
&\leq  \left[\kappa_{1}\nu\frac{(1-t)^{\nu}}{(1-F(t))^{\nu}}
-\kappa_{1}\nu\right]\frac{(1-F(t))^{\nu}}{(1-t)^{\nu}}\\
&=\kappa_{1}\nu\left(\kappa_{1}\nu-\Theta(t)\right)(1-F(t))^{\nu}\frac{(1-F(t))^{\nu}}{(1-t)^{\nu}}\\
&\leq \kappa_{1}\nu\left(\kappa_{1}\nu-\Theta(0)\right)(1-F(t))^{\nu}\frac{(1-F(t))^{\nu}}{(1-t)^{\nu}}\\
&\leq \kappa_{1}\nu \frac{1-(1-\kappa_{1})^{\nu}}{(1-\kappa_{1})^{2\nu}}(1-F(t))^{\nu},
\end{align*}
where the last inequality holds by the fact that $$1-F(t)\le F'(1)(1-t), \quad t\in [0,1).$$ Then, we obtain
\begin{align}\label{up}
\Upsilon_{n}&\leq \kappa_{1}\nu \frac{1-(1-\kappa_{1})^{\nu}}{(1-\kappa_{1})^{2\nu}}\sum\limits_{i=1}^{n}(1-F_{i}(t))^{\nu}\nonumber\\
&\leq \kappa_{1}\nu \frac{1-(1-\kappa_{1})^{\nu}}{(1-\kappa_{1})^{2\nu}}\sum\limits_{i=1}^{n}(1-F_{i}(0))^{\nu}\nonumber\\
&=o(n),\quad n\to \infty.
\end{align}
\eqref{lo} and \eqref{up} yield the desired result.
\end{proof}

Theorem \ref{LFuniform} leads to

\begin{proposition}\label{l2.2} Suppose  (A1) holds. Then
\begin{align}\label{Funi}
(1-F_{n}(t))^{\nu}
=\frac{1+\varepsilon(n,t)}{\kappa_{1}\nu n+(1-t)^{-\nu}},
\end{align}
where $\lim\limits_{n\rightarrow\infty}\varepsilon(n,t)=0$  uniformly for $t\in[0,1)$,  and
\begin{align}\label{varebound}
\varepsilon(n,t)\asymp \frac{\log n}{n},\quad n\rightarrow\infty.
\end{align}
\end{proposition}

\begin{proof}
By  (A1) and \eqref{Theta},
\begin{align*}
\Theta(t)=\kappa_{1}\nu-\frac{1-[1-\kappa_{1}(1-t)^{\nu}]^{\nu}}{(1-t)^{\nu}[1-\kappa_{1}(1-t)^{\nu}]^{2\nu}},
\end{align*}
and one can verify that
\begin{align}\label{Thetalim}
\lim\limits_{t\uparrow1}\frac{1}{\log (1-\kappa_{1}(1-t)^{\nu})}\Theta(t)=\frac{1}{2}\kappa_{1}\nu (3\nu+1).
\end{align}
From \eqref{Defups} and \eqref{Funi},
\begin{align*}
\varepsilon(n,t)=(1-F_{n}(t))^{\nu}\sum\limits_{i=0}^{n-1}\Xi(F_{i}(t)).
\end{align*}
Combining with \eqref{up}, we obtain
\begin{align*}
\varepsilon(n,t)\leq C(1-F_{n}(0))^{\nu}\sum\limits_{i=1}^{n}(1-F_{i}(0))^{\nu}\leq C\frac{\log n}{n}.
\end{align*}
Notice that \eqref{lo} and \eqref{Thetalim},
\begin{align*}
\varepsilon(n,t)&\geq (1-F_{n}(t))^{\nu}\sum\limits_{i=0}^{n-1}\Theta(F_{i}(0))\\
&\geq C(1-F_{n}(t))^{\nu}\sum\limits_{i=1}^{n}\log (1-\kappa_{1}(1-F_{i}(0))^{\nu})\\
&\geq C\frac{\log n}{n}.
\end{align*}
The proof of \eqref{varebound} is completed.
\end{proof}


\section{Life period of $\mathbf{W}$}

In this section, we study the tail probability of life period of $\mathbf{W}$. Define


$$\sigma:=\frac{\kappa_{2}}{\kappa_{1}\nu}.$$

\begin{theorem}\label{thA2}
Suppose   (A1)--(A3) hold.  Then there exist  constants $K_0$--$K_6$ such that as $n\to\infty$,
\begin{align*}
\mathbb{P}(\zeta>n)\sim
\begin{cases}
K_0,& \mbox{if}~\theta<\nu,\\
K_1,& \mbox{if}~\theta=\nu, \sigma>1,\\
K_{2}\log^{-1}n, & \mbox{if}~\theta=\nu, \sigma=1,\\
K_{3}n^{\sigma-1}, & \mbox{if}~\theta=\nu, \sigma+\frac{\delta}{\nu}>1, \sigma<1,\\
K_{4}n^{\sigma-1}\log n, & \mbox{if}~\theta=\nu, \sigma+\frac{\delta}{\nu}=1,\\
K_{5}n^{-\frac{\delta}{\nu}}, & \mbox{if}~\theta=\nu, \sigma+\frac{\delta}{\nu}<1,\\
K_{6}n^{-\frac{\delta}{\nu}}, & \mbox{if}~\theta>\nu, \delta<\nu.
\end{cases}
\end{align*}

\end{theorem}

\begin{remark}

\begin{enumerate}

\item[(1)] By the proof of Theorem~\ref{thA2}, we can see that $K_0$--$K_6$ are expressed  by $\theta$, $\nu$, $\delta$ and $ \kappa_i$, $i=0,1,2$.

\item[(2)] The theorem does not  include all classifications for $\theta$, $\nu$ and $\delta$. For example, we have not obtained the order of $\mathbb{P}(\zeta>n)$ when $\theta>\nu$ and  $\delta\geq \nu$.
    \end{enumerate}
\end{remark}

In the following, we define
\begin{align}\label{D1}
 \gamma_{n}^{(0)}(s):=\prod_{j=0}^{n-1}[B(F_{j}(s))], \quad \gamma_{n}(s):=G_{0}(F_{n}(s))\prod_{j=0}^{n-1}[B(F_{j}(s))],\quad s\in [0,1], \end{align}
and denote $\gamma_{n}^{(0)}(0)$ by $\gamma_{n}^{(0)}$,  $\gamma_{n}(0)$ by $\gamma_{n}$.

\begin{lemma}\label{LB2gamma}
Suppose (A1) and (A2) hold. Then as $n\rightarrow\infty$,
\begin{align*}
 \gamma_{n}^{(0)}\sim
\begin{cases}
c_{0}, & \mbox{if}~\theta>\nu, \\
c_1n^{-\sigma},& \mbox{if}~\theta=\nu,  \\
c_{3}\exp\left\{ -c_2n^{1-\frac{\theta}{\nu}}
 \right\},& \mbox{if}~\theta<\nu,
\end{cases}
\end{align*}
where  $c_{i}$, $i=0,...,3$  are positive constants depending on $\theta$, $\nu$ and $ \kappa_i$, $i=0,1,2$.

\end{lemma}

\begin{proof}
 By \eqref{D1} and Assumption (A2),
\begin{align}\label{1}
 \gamma_{n}^{(0)}
= \exp\left\{- \kappa_{2}\sum_{j=0}^{n-1}(1-F_{j}(0))^{\theta} \right\}.
\end{align}

(i) For $\theta>\nu$,  by \eqref{Funi}, we have
$$\sum_{j=0}^{\infty}(1-F_{j}(0))^{\theta}<\infty,$$
and then
$$ \lim_{n\to \infty} \gamma_{n}^{(0)}=\exp\left\{- \kappa_{2}\sum_{j=0}^{\infty}(1-F_{j}(0))^{\theta} \right\}:=c_{0}.$$

(ii) For $\theta= \nu$, by Proposition~\ref{l2.2},
 \begin{align}\label{es2}
 -\kappa_{2}\sum_{j=0}^{n-1}(1-F_{j}(0))^{\theta}& =- \kappa_{2} \sum_{j=0}^{n-1} \frac{1+\varepsilon(j,0)}{1+\kappa_{1}\nu j}\nonumber\\
 &=-\kappa_{2} \sum_{j=0}^{n-1} \frac{1}{1+\kappa_{1}\nu j}-\kappa_{2}\sum_{j=0}^{n-1} \left(\frac{1+\varepsilon(j,0)}{1+\kappa_{1}\nu j}-\frac{1}{1+\kappa_{1}\nu j}\right)\nonumber\\
 &:=M_1(n)+M_2(n),
\end{align}
where \begin{align*} M_1(n)\sim-\sigma  \log n, \quad n\rightarrow\infty,\end{align*}
and applying \eqref{varebound},
\begin{align*}
\sum_{j=0}^{n-1} \left|\frac{1+\varepsilon(j,0)}{1+\kappa_{1}\nu j}-\frac{1}{1+\kappa_{1}\nu j}\right|\le\sum_{j=0}^{n-1} \frac{|\varepsilon(j,0)|}{1+\kappa_{1}\nu j}\le C\sum_{j=1}^{\infty}\frac{\log j }{j^{2}}<\infty,
\end{align*}
which implies that
$$
 \lim_{n\to \infty} M_2(n)=M_2<\infty.
 $$
By combining above discussions we obtain
$$
\gamma_{n}^{(0)}\sim c_1n^{-\sigma}, \quad n\to \infty,
$$
with $c_1=e^{M_2}$. From the proof of Proposition~\ref{l2.2}, we can see that $c_1$ depends only on $\nu$, $\kappa_1$ and $\kappa_2$.

(iii) For $\theta< \nu$, similarly to  \eqref{es2},
\begin{align*}
-\kappa_{2}\sum_{j=0}^{n-1}(1-F_{j}(0))^{\theta}= -\kappa_{2}\sum_{j=0}^{n-1}\left(\frac{1}{1+\kappa_{1}\nu j}\right)^{\frac{\theta}{\nu}}
+M_{3}(n),
\end{align*}
where
$$M_{3}(n):=-\kappa_{2}\sum_{j=0}^{n-1} \left[\left(\frac{1+\varepsilon(j,0)}{1+\kappa_{1}\nu j}\right)^{\frac{\theta}{\nu}}-\left(\frac{1}{1+\kappa_{1}\nu j}\right)^{\frac{\theta}{\nu}}\right],$$
and
$$\lim\limits_{n\rightarrow\infty}M_{3}(n)=M_{3}<\infty.$$
Then
$$\gamma_{n}^{(0)}\sim c_{3}\exp\left\{ -c_2n^{1-\frac{\theta}{\nu}}
 \right\},\quad n\rightarrow\infty,$$
where $c_2=\kappa_{1}^{-\frac{\theta}{\nu}}\kappa_{2}\nu^{1-\frac{\theta}{\nu}}
 /(\nu-\theta)$, and $c_{3}=e^{M_{3}}$.

We complete the proof.
\end{proof}

The following lemma is from \cite{Z72}.
\begin{lemma}\label{LUA}\cite[Lemma 1]{Z72}  Define
\begin{align*}
u_{k}:=\mathbb{P}(W_{k}>0),\quad U(s):=\sum_{k=0}^{\infty}u_{k}s^{k}.
\end{align*}
Then
\begin{align*}
U(s)=\frac{D(s)}{1-A(s)},
\end{align*}
where
\begin{align*}
D(s):=\frac{1}{\kappa_{0}}\sum_{k=0}^{\infty}d_{k}s^{k}:=\frac{1}{\kappa_{0}}\sum_{k=0}^{\infty}\left(\gamma_{k}^{(0)}- \gamma_{k}\right)s^{k},
\end{align*}
\begin{align*}
A(s):=\sum_{k=0}^{\infty}a_{k}s^{k+1}:=\sum_{k=0}^{\infty}\left(\gamma_{k}^{(0)}-\gamma_{k+1}^{(0)}\right)s^{k+1}.
\end{align*}
\end{lemma}

With above preparation, we now give the proof of  Theorem \ref{thA2}.

\textbf{Proof of Theorem \ref{thA2}.} 
First, we consider the asymptotic behavior of $1-A(s)$ as $s\to 1^{-}$. According to Assumption (A2),
\begin{align*}
a_{k}=\gamma_{k}^{(0)}[1-B(F_{k}(0))]
\sim \kappa_{2}\left(\frac{1}{\kappa_{1}\nu}\right)^{\theta/\nu}k^{-\frac{\theta}{\nu}}\gamma_{k}^{(0)},\quad k\rightarrow\infty.
\end{align*}
Together with Lemma \ref{LB2gamma}, we obtain
\begin{align}\label{ekak}
ka_{k}\sim
\begin{cases}
c_4k^{1-\frac{\theta}{\nu}}
\exp\left\{-c_2k^{1-\frac{\theta}{\nu}}
\right\}
,& \mbox{if}~\theta<\nu,  \\
c_5 k^{-\sigma},& \mbox{if}~\theta=\nu,  \\
c_6k^{1-\frac{\theta}{\nu}}, & \mbox{if}~\theta>\nu,
\end{cases}
\end{align}
where $c_i$, $i=4,5,6$ are positive constants depending on $\theta$, $\nu$ and $ \kappa_i$, $i=0,1,2$.

(i) If $\theta=\nu$ and $\sigma>1$ or $\theta<\nu$, by Lemma \ref{LB2gamma} and  \eqref{ekak},
\begin{align*}
A(1)=1-\lim\limits_{n\rightarrow\infty}\gamma_{n}^{(0)}=1,\quad A'(1)=\sum\limits_{k=0}^{\infty}(k+1)a_k<\infty,
\end{align*}
and therefore,
\begin{align}\label{eA}
1-A(s) \sim A'(1)(1-s),\quad s\to 1^{-}.
\end{align}

(ii) If $\theta=\nu$ and $\sigma=1$, by \eqref{ekak} we have
\begin{align*}
\sum\limits_{k=0}^{n}ka_{k}\sim  c_5\log n,\quad n\rightarrow\infty.
\end{align*}
It follows from Tauberian Theorem~\cite[Chapter XIII5, Theorem 5]{Feller2} that
\begin{align*}
A'(s)\sim  -c_5\log(1-s), \quad s\to 1^{-}.
\end{align*}
Then
\begin{align}\label{eAs}
1-A(s)=\int_{s}^{1}A'(t) \mathrm{d}t
\sim  -c_5(1-s)\log(1-s),\quad s\to 1^{-}.
\end{align}

(iii) If $\theta=\nu$ and $\sigma<1$, also by \cite[Chapter XIII5, Theorem 5]{Feller2}, we get
\begin{align}\label{eAs1}
1-A(s)
\sim  c_5\sigma^{-1}\Gamma(1-\sigma)(1-s)^{\sigma},\quad s\rightarrow1^{-}.
\end{align}

(iv) If $\theta>\nu$, then
\begin{align}\label{eAs2}
A(1)=1-\lim\limits_{n\rightarrow\infty}\gamma_{n}^{(0)}=1-c_{0}\in(0,1).
\end{align}
Collecting \eqref{eA}--\eqref{eAs2}, we conclude that as $s\to 1^{-}$,
\begin{align}\label{Asbehavior}
1-A(s) \sim
\begin{cases}
A'(1)(1-s), &\theta=\nu,\sigma>1,~\mbox{or}~\theta<\nu,\\
-c_5(1-s)\log (1-s), &\theta=\nu, \sigma=1,\\
c_5\sigma^{-1}\Gamma(1-\sigma)(1-s)^{\sigma}, &\theta=\nu, \sigma<1,\\
c_{0}, &\theta>\nu.
\end{cases}
\end{align}

Next, we consider the asymptotic behavior of $D(s)$ as $s\to 1^{-}$. According to Assumption (A3),
\begin{align*}
d_{k}=\gamma_{k}^{(0)}[1-G_{0}(F_{k}(0))]
\sim \kappa_{0}\left(\frac{1}{\kappa_{1}\nu}\right)^{\delta/\nu}k^{-\frac{\delta}{\nu}}\gamma_{k}^{(0)},\quad k\rightarrow\infty.
\end{align*}
As we do for $1-A(s)$, applying \cite[Chapter XIII5, Theorem 5]{Feller2}, we have as $s\to 1^{-}$,
\begin{align}\label{Dsbehavior}
D(s) \sim
\begin{cases}
-c_7\log(1-s), & \theta=\nu, \sigma+\frac{\delta}{\nu}=1,\\
c_8 (1-s)^{\sigma+\frac{\delta}{\nu}-1}, &\theta=\nu, \sigma+\frac{\delta}{\nu}<1,\\
-c_9\log(1-s), &\theta>\nu, \delta=\nu,\\
c_{10} (1-s)^{\frac{\delta}{\nu}-1}, &\theta>\nu, \delta<\nu,\\
c_{11}, &\mbox{otherwise},\\
\end{cases}
\end{align}
where $c_i$, $i=7,\cdots,11$ are positive constants depending on $\theta$, $\nu$, $\delta$ and $ \kappa_i$, $i=0,1,2$. Noticing  the sequence $\{u_{n},n\geq 0\}$ is decreasing in $n$, we end the proof by \eqref{Asbehavior}--\eqref{Dsbehavior}, and \cite[Chapter XIII5, Theorem 5]{Feller2}.\qed

\section{Scaling limit of $\mathbf{Z}$}

With the help of Theorem~\ref{LFuniform}, we obtain
\begin{theorem}\label{Lgamma1}
Suppose  (A1) and (A2) hold.  If  $\theta=\nu$, then for each $t>0$,
\begin{align}\label{gamma0}
\lim_{n\to \infty}\sup\limits_{k\leq n}\bigg|\left(1+\frac{k}{n}t^{\nu}\right)^\sigma\gamma_{k}^{(0)}\left(e^{-t(1-F_{n}(0))}\right)
-1\bigg|
=0.
\end{align}
\end{theorem}

\begin{proof}
According to the definition of $\gamma_{k}^{(0)}$,
\begin{align*}
\gamma_{k}^{(0)}\left(e^{-t(1-F_{n}(0))}\right)
= \exp\left\{-\kappa_{2}\sum_{j=0}^{k-1}\left(1-F_{j}\left(e^{-t(1-F_{n}(0))}\right)\right)^{\theta} \right\}.
\end{align*}
Noticing that  $\theta=\nu$, we have
\begin{align}\label{gamma1}
& \left(1+\frac{k}{n}t^{\nu}\right)^\sigma\gamma_{k}^{(0)}\left(e^{-t(1-F_{n}(0))}\right)\nonumber\\
&=\left(1+\frac{k}{n}t^{\nu}\right)^\sigma\exp\left\{-\kappa_{2}\sum_{j=0}^{k-1}\frac{1+\varepsilon\left(j,e^{-t(1-F_{n}(0))}\right)}{\kappa_{1}\nu j+\left(1-e^{-t(1-F_{n}(0))}\right)^{-\nu}} \right\}\nonumber\\
&=\exp\left\{-\kappa_{2}\left[\left(\Phi_1(k,n,t)-\frac{1}{\kappa_{1}\nu}\log \left(1+\frac{k}{n}t^{\nu}\right)\right)+\Phi_2(k,n,t)+\Phi_3(k,n,t)\right] \right\},
\end{align}
where
\begin{align*}
& \Phi_1(k,n,t):=\sum_{j=0}^{k-1}\frac{1}{\kappa_{1}\nu j+t^{-\nu}(1-F_{n}(0))^{-\nu}},\\
& \Phi_2(k,n,t):=\sum_{j=0}^{k-1}\left(\frac{1}{\kappa_{1}\nu j+\left(1-e^{-t(1-F_{n}(0))}\right)^{-\nu}}-\frac{1}{\kappa_{1}\nu j+t^{-\nu}(1-F_{n}(0))^{-\nu}}\right),\\
&  \Phi_3(k,n,t):= \sum_{j=0}^{k-1}\frac{\varepsilon\left(j,e^{-t(1-F_{n}(0))}\right)}{\kappa_{1}\nu j+\left(1-e^{-t(1-F_{n}(0))}\right)^{-\nu}}.
\end{align*}
Observing
\begin{align*}
\Phi_{1}(k,n,t)\geq \int_{0}^{k}\frac{1}{\kappa_{1}\nu x+t^{-\nu}(1-F_{n}(0))^{-\nu}}dx,
\end{align*}
and
\begin{align*}
\Phi_{1}(k,n,t)\leq \int_{0}^{k-1}\frac{1}{\kappa_{1}\nu x+t^{-\nu}(1-F_{n}(0))^{-\nu}}dx
+t^{\nu}(1-F_{n}(0))^{\nu}.
\end{align*}
Together with \eqref{Funi}, we obtain
\begin{align*}
\sup_{k\le n}\left|\Phi_{1}(k,n,t)- \frac{1}{\kappa_{1}\nu}\log \left(1+\frac{k}{n}t^{\nu}\right)\right|\to 0,\quad n\rightarrow\infty.
\end{align*}
Setting $x_{n,t}=t(1-F_{n}(0))$.   Recalling Proposition \ref{l2.2}, we have
\begin{align*}
x_{n,t}\sim t(\kappa_{1}\nu)^{-\frac{1}{\nu}}n^{-\frac{1}{\nu}}, \quad n\rightarrow\infty,
\end{align*}
and then for $\nu<1$,
\begin{align*}
\sup_{k\le n}|\Phi_{2}(k,n,t)|&\leq \sup_{k\le n}\sum_{j=0}^{k-1}\frac{\left(1-e^{-x_{n,t}}\right)^{-\nu}-x_{n,t}^{-\nu}}
{[\kappa_{1}\nu j+\left(1-e^{-x_{n,t}}\right)^{-\nu}][\kappa_{1}\nu j+x_{n,t}^{-\nu}]}\\
&\leq \frac{\left(1-e^{-x_{n,t}}\right)^{-\nu}-x_{n,t}^{-\nu}}{(\kappa_{1}\nu)^{2}}\sum_{j=1}^{\infty}\frac{1}{j^{2}}\\
&\leq C \left(\left(1-e^{-x_{n,t}}\right)^{-\nu}-x_{n,t}^{-\nu}\right)\\
&\rightarrow0
\end{align*}
as $n\rightarrow\infty$. If $\nu=1$,
\begin{align*}
\sup_{k\le n}|\Phi_{2}(k,n,t)|\leq \sum_{j=0}^{n-1}\big| x_{n,t}-1+e^{-x_{n,t}}\big|
\le n \big| x_{n,t}-1+e^{-x_{n,t}}\big|\rightarrow0, \quad n\rightarrow\infty.
\end{align*}
Finally, applying \eqref{varebound}, one can verify that
\begin{align*}
\lim\limits_{n\rightarrow\infty}\sup_{k\leq n}|\Phi_{3}(k,n,t)|=0.
\end{align*}

We end the proof.
\end{proof}

Observe that\begin{align*}
\mathbb{E}\left[e^{-t(1-F_{n}(0))Z_{n}}\right]&=G_{0}\left(F_{n}\left(e^{-t(1-F_{n}(0))}\right)\right)
\prod_{j=0}^{n-1}\left[B\left(F_{j}\left(e^{-t(1-F_{n}(0))}\right)\right)\right]\\
&=\left[1-\kappa_{0}\left(1-F_{n}\left(e^{-t(1-F_{n}(0))}\right)\right)^{\delta}\right]
\gamma_{n}^{(0)}\left(e^{-t(1-F_{n}(0))}\right).
\end{align*}
Letting $k=n$ in Theorem~\ref{Lgamma1}, together with Proposition~\ref{l2.2},  we then obtain

\begin{corollary}\label{s1}
Suppose  (A1)--(A3) hold. If $\theta=\nu$,  then for each $t>0$,
\begin{align*}
\lim_{n\to \infty}\mathbb{E}\left[e^{-t(1-F_{n}(0))Z_{n}}\right]=(1+ t^{\nu})^{-\sigma}.
 \end{align*}
\end{corollary}

Similarly to Theorem~\ref{Lgamma1}, we have
\begin{theorem}\label{Lgamma2}
Suppose  (A1) and (A2) hold. If $\theta<\nu$,  then for each $t>0$,
\begin{align}\label{gamma00}
\lim_{n\to \infty}\sup\limits_{k\leq n}\bigg|e^{\kappa_{2}t^{\theta}\frac{k}{n}}\cdot \gamma_{k}^{(0)}\left(e^{-tn^{-\frac{1}{\theta}}}\right)-1\bigg|
=0.
\end{align}
\end{theorem}

Setting $k=n$ in Theorem~\ref{Lgamma2} yields

\begin{corollary}\label{s2}
 Suppose  (A1)--(A3) hold. If $\theta<\nu$,  then for each $t>0$,
\begin{align*}
\lim_{n\to \infty}\mathbb{E}\left[e^{-tZ_{n}/n^{\frac{1}{\theta}}}\right]=e^{-\kappa_{2}t^{\theta}}.
 \end{align*}
\end{corollary}

Finally, we state the result for the case $\theta>\nu$ :

\begin{theorem}\label{Llimit3}
 Suppose  (A1)--(A3) hold. If $\theta>\nu$, then
\begin{align*}
Z_{n}\stackrel{d}\rightarrow \psi.
\end{align*}
\end{theorem}

\begin{proof}
It is known from \cite{P75} that $\mathbf{Z}$ has a limit stationary distribution iff
\begin{align*}
\sum\limits_{n=1}^{\infty}(1-B(F_{n}(0)))<\infty,
\end{align*}
which holds by Assumption (A2) and  $\theta>\nu$.

\end{proof}

\section{Conditional limit theorems of $\mathbf{W}$}

For convenience, in the following we denote
\begin{align}\label{uns}
u_n=\mathbb{P}(W_{n}>0)\sim n^{-\alpha}\mathcal{L}_{\ast}(n),\quad n\rightarrow\infty,
\end{align}
where $\alpha\in[0,1)$ and $\mathcal{L}_{\ast}$ are specified by Theorem~\ref{thA2}, according to the relation of $\sigma, \nu,\theta,\delta$.

\subsection{Main results}

\begin{theorem}\label{thA3}
Suppose  (A1)--(A3) hold  and $\theta<\nu$. Then
\begin{align*}
\lim\limits_{n\rightarrow\infty}\mathbb{E}\left(e^{-sn^{-1/\theta}W_{n}}\big|W_{n}>0\right)= e^{-\kappa_{2}s^{\theta}}.
\end{align*}
\end{theorem}

\begin{theorem}\label{thA4}
 Suppose  (A1)--(A3) hold and $\theta=\nu$. \\
(a)~If  $\sigma\ge 1$, then
\begin{align*}
\lim\limits_{n\rightarrow\infty}\mathbb{E}\left(e^{-s(1-F_{n}(0))W_{n}}\big|W_{n}>0\right)= \left(1+s^{\nu}\right)^{-\sigma};
\end{align*}
(b)~If $\sigma<1$, then
\begin{align*}
\lim\limits_{n\rightarrow\infty}\mathbb{E}\left(e^{-s(1-F_{n}(0))W_{n}}\big|W_{n}>0\right)=
\Lambda(s,\delta,\nu),
\end{align*}
where
\begin{align*}
\Lambda(s,\delta,\nu)=
\begin{cases}
1-\frac{\kappa_{0}}{K_{5}}\left(\frac{1}{\kappa_{1}\nu}\right)^{\delta/\nu}
\frac{s^{\delta}}{\left(1+s^{\nu}\right)^{\sigma+\delta/\nu}}
-\sigma s^{\nu}
\int_{0}^{1}\frac{(1-x)^{-\delta/\nu}}{(1+s^{\nu}x)^{\sigma+1}}dx, & \sigma<1-\delta/\nu,\\
1
-\sigma s^{\nu}
\int_{0}^{1}\frac{(1-x)^{\sigma-1}}{(1+s^{\nu}x)^{\sigma+1}}dx, & \sigma\geq 1-\delta/\nu,\\
\end{cases}
\end{align*}
with $K_5$ given by Theorem~\ref{thA2}. 
\end{theorem}

\subsection{Proofs}

For the convenience of description, denote $G_{n}(s):=\mathbb{E}(s^{W_{n}})$. By \eqref{DefWn},
\begin{align*}
G_{n}(s)=\sum\limits_{k=0}^{n-1}G_{n-1-k}(0)\left(\gamma_{k}^{(0)}(s)-\gamma_{k+1}^{(0)}(s)\right)
+G_{0}(F_{n}(s))\gamma_{n}^{(0)}(s).
\end{align*}
Then,
\begin{align}\label{Wndecom}
1-\mathbb{E}\left(t^{W_{n}}\bigg|W_{n}>0\right)&=\frac{\gamma_{n}^{(0)}(t)(1-G_{0}(F_n(t)))}{u_{n}}
+\sum_{k=1}^{n}\frac{u_{n-k}}{u_{n}}\left(\gamma_{k-1}^{(0)}(t)-\gamma_{k}^{(0)}(t)\right) \nonumber\\
&:=\Xi_{1}(t,n) + \Xi_{2}(t,n).
\end{align}

\textbf{Proof of Theorem \ref{thA3}.}
By Theorem \ref{thA2}, $u_{n}\rightarrow K_0$ ($n\rightarrow\infty$).  By (A3) and $1-F_n(t)\le F_n'(1)(1-t) (t\in (0,1))$,  we obtain as $n\rightarrow\infty$,
\begin{align}\label{Xi1}
\Xi_{1}\big(e^{-sn^{-1/\theta}},n\big)
=\frac{\kappa_{0}\gamma_{n}^{(0)}\left(e^{-sn^{-1/\theta}}\right)}{u_{n}}
\left(1-F_n\left(e^{-sn^{-1/\theta}}\right)\right)^{\delta}
 \leq C\left(1-e^{-sn^{-1/\theta}}\right)^{\delta}\to 0.
\end{align}
For any $\varepsilon>0$, there exists $N>0$, such that for $k>N$, $K_0<u_k<K_0+\varepsilon$. Then
\begin{align}\label{xi22}
\Xi_{2}\big(e^{-sn^{-1/\theta}},n\big)&\geq \frac{K_0}{K_0+\varepsilon}\sum_{k=1}^{n-N}\left[\gamma_{k-1}^{(0)}\left(e^{-sn^{-1/\theta}}\right)-\gamma_{k}^{(0)}\left(e^{-sn^{-1/\theta}}\right)\right] \nonumber\\ &~~~+\sum_{k=n-N+1}^{n}\frac{u_{n-k}}{u_{n}}\left[\gamma_{k-1}^{(0)}\left(e^{-sn^{-1/\theta}}\right)-\gamma_{k}^{(0)}\left(e^{-sn^{-1/\theta}}\right)\right] \nonumber\\
&:=\frac{K_0}{K_0+\varepsilon}S_{1}(s,n,\theta)+S_{2}(s,n,\theta),
\end{align}
and
\begin{align}\label{xi222}
\Xi_{2}\big(e^{-sn^{-1/\theta}},n\big)\leq \frac{K_0+\varepsilon}{K_0}S_{1}(s,n,\theta)+S_{2}(s,n,\theta).
\end{align}
On the one hand, by $1-B(F_k(t))\leq  C(1-t)^{\theta} (t\in (0,1))$, we have
\begin{align}\label{s22}
S_{2}(s,n,\theta)&\leq C\sum_{k=n-N+1}^{n}\gamma_{k-1}^{(0)}\left(e^{-sn^{-1/\theta}}\right)
\left[1-B\left(F_{k-1}\left(e^{-sn^{-1/\theta}}\right)\right)\right]\nonumber\\
&\leq C \sum_{k=n-N+1}^{n}\left(1-e^{-sn^{-1/\theta}}\right)^{\theta}\nonumber\\
&\leq CN/n\rightarrow0, \quad n\rightarrow\infty.
\end{align}
On the other hand, by Theorem \ref{Lgamma2},
\begin{align}\label{s11}
S_{1}(s,n,\theta)=1-\gamma_{n-N}^{(0)}\left(e^{-sn^{-1/\theta}}\right)
\rightarrow 1-e^{-\kappa_{2}s^{\theta}},
\end{align}
as $n\rightarrow\infty$. Collecting \eqref{xi22}--\eqref{s11},  letting $n\to \infty$ first, and then $\varepsilon\downarrow0$,  we get
\begin{align*}
\Xi_{2}\big(e^{-sn^{-1/\theta}},n\big)\to 1-e^{-\kappa_{2}s^{\theta}}.
\end{align*}
Recalling \eqref{Wndecom} and \eqref{Xi1}, we end the proof.\qed

\textbf{Proof of Theorem \ref{thA4}.}

First, we  consider the asymptotic behavior of $\Xi_{1}\big(e^{-s(1-F_{n}(0))},n\big)$. Using \eqref{uns},
\begin{align*}
\Xi_{1}\big(e^{-s(1-F_{n}(0))},n\big)&=\frac{\kappa_{0}\gamma_{n}^{(0)}\left(e^{-s(1-F_{n}(0))}\right)}{u_{n}}\left(1-F_n\left(e^{-s(1-F_{n}(0))}\right)\right)^{\delta}\\
&\leq Cn^{\alpha}\mathcal{L}_{\ast}(n)^{-1}\left(1-F_{n}(0)\right)^{\delta}\\
&\leq Cn^{\alpha-\delta/\nu}\mathcal{L}_{\ast}(n)^{-1}.
\end{align*}

 (i) For $\sigma\ge  1-\delta/\nu$, according to Theorem \ref{thA2}, we have
 \begin{align}\label{5.3}
\Xi_{1}\big(e^{-s(1-F_{n}(0))},n\big)\to 0,\quad n\to\infty.\end{align}

(ii)  For $\sigma< 1-\delta/\nu$, by Theorem \ref{thA2} we have $u_n\sim K_5n^{-\delta/\nu}$. Using  Theorem~\ref{Lgamma1} and \eqref{Funi}, we can obtain
\begin{align}\label{Xi1big}
\Xi_{1}\big(e^{-s(1-F_{n}(0))},n\big)\rightarrow \frac{\kappa_{0}}{K_{5}}\left(\frac{1}{\kappa_{1}\nu}\right)^{\delta/\nu}
\frac{s^{\delta}}{\left(1+s^{\nu}\right)^{\sigma+\delta/\nu}},\quad n\to \infty.
\end{align}

Next, we consider  $\Xi_{2}\big(e^{-s(1-F_{n}(0))},n\big)$.

(a) $\sigma\geq 1$. In this case $\alpha=0$. For any $0<\lambda<1$, define the largest integer  that less than or equal to $\lambda n$ by $[\lambda n]$. Then
\begin{align}\label{xi2}
\Xi_{2}\big(e^{-s(1-F_{n}(0))},n\big)&=\sum_{k\leq [\lambda n]}\frac{u_{n-k}}{u_{n}}\left[\gamma_{k-1}^{(0)}\left(e^{-s(1-F_{n}(0))}\right)-\gamma_{k}^{(0)}\left(e^{-s(1-F_{n}(0))}\right)\right] \nonumber\\
&~~+\sum_{[\lambda n]<k\leq n}\frac{u_{n-k}}{u_{n}}\left[\gamma_{k-1}^{(0)}\left(e^{-s(1-F_{n}(0))}\right)-\gamma_{k}^{(0)}\left(e^{-s(1-F_{n}(0))}\right)\right] \nonumber\\
&:=S_{3}(s,n)+S_{4}(s,n).
\end{align}
By   $1-B(F_k(t))\leq  C(1-t)^{\theta} (t\in (0,1))$  and $\theta=\nu$,
\begin{align*}
S_{4}(s,n)
&\leq C\sum_{[\lambda n]<k\leq n}\frac{u_{n-k}}{u_{n}}\gamma_{k-1}^{(0)}\left(e^{-s(1-F_{n}(0))}\right)
\left(1-e^{-s(1-F_{n}(0))}\right)^{\theta}\\
&\le C \sum_{[\lambda n]<k\leq n}\frac{u_{n-k}}{u_{n}}
\left(1-e^{-s(1-F_{n}(0))}\right)^{\nu}\\
&\leq \frac{C}{nu_{n}}\sum_{k=1}^{n-[\lambda n]}u_{k}.
\end{align*}
Recalling \eqref{uns} and applying  Theorem \ref{thA2}, as $n\to \infty$,
\begin{align}\label{un10}
u_n\sim
\begin{cases}
K_1, & \sigma>1,\\
K_2\log ^{-1} n, & \sigma=1.\\
\end{cases}
\end{align}
Then
\begin{align*}
\limsup\limits_{n\rightarrow\infty}S_{4}(s,n)
\leq  C(1-\lambda).
\end{align*}
Moreover, noticing that the sequence $\{u_{n},n\geq 1\}$ is decreasing, we obtain
\begin{align*}
S_{3}(s,n)
&\leq \frac{u_{n-[\lambda n]}}{u_{n}}\sum_{k\leq [\lambda n] }\left[\gamma_{k-1}^{(0)}\left(e^{-s(1-F_{n}(0))}\right)-\gamma_{k}^{(0)}\left(e^{-s(1-F_{n}(0))}\right)\right] ,
\end{align*}
and
\begin{align*}
S_{3}(s,n)
&\geq \sum_{k\leq [\lambda n]}\left[\gamma_{k-1}^{(0)}\left(e^{-s(1-F_{n}(0))}\right)-\gamma_{k}^{(0)}\left(e^{-s(1-F_{n}(0))}\right)\right] .
\end{align*}
Applying  Theorem~\ref{Lgamma1}, as $n\rightarrow\infty$, we have
\begin{align*}
\sum_{k\leq [\lambda n]}\left[\gamma_{k-1}^{(0)}\left(e^{-s(1-F_{n}(0))}\right)-\gamma_{k}^{(0)}\left(e^{-s(1-F_{n}(0))}\right)\right]
\rightarrow 1-(1+\lambda s^{\nu})^{-\sigma}.
\end{align*}
For any  $0<\lambda<1$, by \eqref{un10} we have
\begin{align*}\frac{u_{n-[\lambda n]}}{u_{n}}  \rightarrow1,\quad n\rightarrow\infty.
\end{align*}
 Therefore, letting $n\rightarrow\infty$ first, and then $\lambda\uparrow1$ in (\ref{xi2}), we obtain
 \begin{align}\label{xi20}
\Xi_{2}\big(e^{-s(1-F_{n}(0))},n\big)
\to 1-(1+ s^{\nu})^{-\sigma}.\end{align}
Combining \eqref{Wndecom}, \eqref{5.3} and \eqref{xi20}, we complete the proof of (a).

(b) $\sigma<1$.  In this case, we use  a composition of $\Xi_{2}\big(e^{-s(1-F_{n}(0))},n\big)$ as follows: For any $\lambda\in(0,1/2)$,
\begin{align}\label{xi222}
\Xi_{2}\big(e^{-s(1-F_{n}(0))},n\big)&=\sum_{k\leq [\lambda n]}\frac{u_{n-k}}{u_{n}}\left[\gamma_{k-1}^{(0)}\left(e^{-s(1-F_{n}(0))}\right)-\gamma_{k}^{(0)}\left(e^{-s(1-F_{n}(0))}\right)\right] \nonumber\\
&~~+\sum_{[\lambda n]<k\leq [(1-\lambda)n]}\frac{u_{n-k}}{u_{n}}\left[\gamma_{k-1}^{(0)}\left(e^{-s(1-F_{n}(0))}\right)-\gamma_{k}^{(0)}\left(e^{-s(1-F_{n}(0))}\right)\right]\nonumber\\
&~~~+\sum_{[(1-\lambda) n]<k\leq n}\frac{u_{n-k}}{u_{n}}\left[\gamma_{k-1}^{(0)}\left(e^{-s(1-F_{n}(0))}\right)-\gamma_{k}^{(0)}\left(e^{-s(1-F_{n}(0))}\right)\right]\nonumber\\
&:=S_{5}(s,n)+S_{6}(s,n)+S_{7}(s,n).
\end{align}
First, recalling \eqref{uns} and applying  Theorem \ref{thA2}, as $n\to \infty$,
\begin{align}\label{un11}
u_n\sim
\begin{cases}
K_3n^{\sigma-1}, & \sigma+\frac{\delta}{\nu}>1,\\
K_4n^{\sigma-1}\log   n, & \sigma+\frac{\delta}{\nu}=1,\\
K_5n^{-\frac{\delta}{\nu}},& \sigma+\frac{\delta}{\nu}<1.
\end{cases}
\end{align}
As arguing for $S_{4}(s,n)$, we have
\begin{align*}
S_{5}(s,n)&=\sum_{k\leq [\lambda n]}\frac{u_{n-k}}{u_{n}}\gamma_{k-1}^{(0)}\left(e^{-s(1-F_{n}(0))}\right)
\left[1-B\left(F_{k-1}\big(e^{-s(1-F_{n}(0))}\big)\right)\right]\\
&\leq C\sum_{k\leq [\lambda n]}\frac{u_{n-k}}{u_{n}}
\left(1-e^{-s(1-F_{n}(0))}\right)^{\theta}\\
&\leq \frac{C}{nu_{n}}\sum_{k=n-[\lambda n]}^{n}u_{k},
\end{align*}
and
$$
\limsup\limits_{n\rightarrow\infty}S_{5}(s,n)\leq C[1-(1-\lambda)^{1-\alpha}].$$
Similarly,
\begin{align*}
\limsup\limits_{n\rightarrow\infty}S_{7}(s,n)\leq C\lambda^{1-\alpha}.
\end{align*}
Last, by (A2) and Taylor's expansion $$e^{-x}=1-x+ \frac{\xi_x}{2}x^2, \quad x,  \xi_x\in (0,1), $$ we have that
\begin{align*}
S_{6}(s,n)&=\sum_{[\lambda n]<k\leq [(1-\lambda) n]}\frac{u_{n-k}}{u_{n}}\gamma_{k-1}^{(0)}\left(e^{-s(1-F_{n}(0))}\right)
\left[1-B\left(F_{k-1}\left(e^{-s(1-F_{n}(0))}\right)\right)\right]\\
&=\sum_{[\lambda n]<k\leq [(1-\lambda) n]}\frac{u_{n-k}}{u_{n}}\left[\gamma_{k-1}^{(0)}\left(e^{-s(1-F_{n}(0))}\right)
-\left(1+\frac{k-1}{n}s^{\nu}\right)^{-\sigma}\right]
\left[1-B\left(F_{k-1}\left(e^{-s(1-F_{n}(0))}\right)\right)\right]\\
&~~~+\sum_{[\lambda n]<k\leq [(1-\lambda) n]}\frac{u_{n-k}}{u_{n}}\left(1+\frac{k-1}{n}s^{\nu}\right)^{-\sigma}\kappa_{2}\left(1-F_{k-1}\big(e^{-s(1-F_{n}(0))}\big)\right)^{\theta}
\\
&~~~~~~+\sum_{[\lambda n]<k\leq [(1-\lambda) n]}\frac{u_{n-k}}{u_{n}}\left(1+\frac{k-1}{n}s^{\nu}\right)^{-\sigma}
\varphi(s,k,n,\theta)\\
&:=\Delta_{1}(s,n,\theta)+\Delta_{2}(s,n,\theta)+\Delta_{3}(s,n,\theta),
\end{align*}
where
\begin{align*}
\varphi(s,k,n,\theta)\le C\left[1-F_{k-1}\left(e^{-s(1-F_{n}(0))}\right)\right]^{2\theta}.
\end{align*}
Then as $n\to \infty$, by $1-F_k(t)\le F_k'(1)(1-t) (t\in (0,1))$  and $\theta=\nu$,
\begin{align*}
|\Delta_{3}(s,n,\theta)|&\leq \sum_{[\lambda n]<k\leq [(1-\lambda) n]}\frac{u_{n-k}}{u_{n}}\left(1+\frac{k-1}{n}s^{\nu}\right)^{-\sigma}
|\varphi(s,k,n,\theta)|\\
&\leq C\sum_{[\lambda n]<k\leq [(1-\lambda) n]}\frac{u_{n-k}}{u_{n}}\left(1+\frac{k-1}{n}s^{\nu}\right)^{-\sigma}\big[1-F_{k-1}\big(e^{-s(1-F_{n}(0))}\big)\big]^{2\theta}\\
&\leq \frac{C}{n^{2}}\sum_{[\lambda n]<k\leq [(1-\lambda) n]}\frac{u_{n-k}}{u_{n}}\\
&\rightarrow0,
\end{align*}
where the last step is by \eqref{un11}. Applying  Theorem \ref{Lgamma1},
\begin{align*}
\limsup\limits_{n\rightarrow\infty}|\Delta_{1}(s,n,\theta)|=0.
\end{align*}
Recalling $\theta=\nu$ and applying Proposition~\ref{l2.2}, we obtain
\begin{align*}
\lim\limits_{n\rightarrow\infty}\Delta_{2}(s,n,\theta)=\sigma s^{\nu}
\int_{\lambda}^{1-\lambda}(1-x)^{-\alpha}(1+s^{\nu}x)^{-\sigma-1}dx.
\end{align*}
Letting $n\rightarrow\infty$ first, and then $\lambda\downarrow0$ in \eqref{xi222}, together with above discussions on $S_{5}(s,n)$, $S_{6}(s,n)$, $S_{7}(s,n)$, we have
\begin{align*}
\Xi_{2}\big(e^{-s(1-F_{n}(0))},n\big)\rightarrow \sigma s^{\nu}
\int_{0}^{1}(1-x)^{-\alpha}(1+s^{\nu}x)^{-\sigma-1}dx.
\end{align*}
Combining with \eqref{Wndecom}, \eqref{5.3} and \eqref{Xi1big}, we obtain the desired result. \qed
\bigskip

{\bf Acknowledgement.}  The first author is supported by China Postdoctoral Science Foundation (Grant No. 2020M680269). The second author is supported by National Natural Science Foundation of China (Grant No. 11871103).

\bigskip

\end{document}